%% file: ocp_arxiv.tex
\documentclass[10pt,draftcls,onecolumn]{IEEEtran}
\usepackage{epsfig,amssymb,latexsym,color,pifont,colordvi,multicol}
\usepackage[cmex10]{amsmath}
\usepackage{subfigure}

\definecolor{backgrey}{rgb}{0.86,0.86,0.86}
\definecolor{dblue}{rgb}{0,0.0,0.5}
\definecolor{dred}{rgb}{0.4,0.2,0}
\definecolor{dgreen}{rgb}{0.0,0.5,0}

\newcommand{\captionfonts}{\small}
\makeatletter  % Allow the use of @ in command names
\long\def\@makecaption#1#2{%
  \vskip\abovecaptionskip
  \sbox\@tempboxa{{\captionfonts #1: #2}}%
  \ifdim \wd\@tempboxa >\hsize
    {\captionfonts #1: #2\par}
  \else
    \hbox to\hsize{\hfil\box\@tempboxa\hfil}%
  \fi
  \vskip\belowcaptionskip}
\makeatother   % Cancel the effect of \makeatletter
\newtheorem{theorem}{Theorem}

\newtheorem{assumption}[theorem]{Assumption}

\newtheorem{remark}[theorem]{Remark}

\newtheorem{definition}[theorem]{Definition}

\newtheorem{lemma}[theorem]{Lemma}

\newenvironment{proof}[1][Proof]{\textbf{#1.} }{\ \hspace*{\fill} \rule{0.5em}{0.5em}}

\usepackage{wrapfig}
\definecolor{backgrey}{rgb}{0.86,0.86,0.86}
\definecolor{dblue}{rgb}{0,0.0,0.5}
\definecolor{dred}{rgb}{0.4,0.2,0}
\definecolor{dgreen}{rgb}{0.0,0.5,0}
\def\oP{P^{1}}

\def\R{\mathbb{R}}

\newcommand{\sub}[2]{(#1)_{(#2)}}
\makeatletter  % Allow the use of @ in command names
\long\def\@makecaption#1#2{%
  \vskip\abovecaptionskip
  \sbox\@tempboxa{{\captionfonts #1: #2}}%
  \ifdim \wd\@tempboxa >\hsize
    {\captionfonts #1: #2\par}
  \else
    \hbox to\hsize{\hfil\box\@tempboxa\hfil}%
  \fi
  \vskip\belowcaptionskip}
\makeatother   % Cancel the effect of \makeatletter
\title{Optimal stabilization using Lyapunov measures}
\author{\quad Arvind Raghunathan \quad Umesh Vaidya %\quad Arvind U. Raghunathan %
\thanks{A. Raghunathan is with Mitsubishi Electric Research Laboratories in Cambridge MA 02139
        {\tt\small raghunathan@merl.com}}%

\thanks{U. Vaidya is with the Department of Electrical \& Computer Engineering,
Iowa State University, Ames, IA 50011
        {\tt\small ugvaidya@iastate.edu}}%
        }
\begin{document}
\maketitle \thispagestyle{empty} \pagestyle{empty}

\begin{abstract}
\input{abstract_ocp}
\end{abstract}


\section{Introduction}
Stability analysis and stabilization of nonlinear systems are two
of the most important, extensively studied problems in control
theory. Lyapunov functions are used for stability analysis and
control Lyapunov functions (CLF) are used in the design of
stabilizing feedback controllers.
%Though Sontag \cite{Sontag_89} provides
%a formula for constructing a stabilizing controller under assumption of a
%CLF but there is no systematic procedure to construct one.
Under
the assumption of detectability and stabilizability of the
nonlinear system, a positive valued optimal cost function of an
optimal control problem (OCP) can also be used as a control Lyapunov
function. The optimal controls of OCP are obtained as the solution of
the Hamilton Jacobi Bellman (HJB) equation.  The HJB equation is a
nonlinear partial differential equation and one must resort to
approximate numerical schemes for its solution. Numerical schemes
typically discretize the state-space; hence, the resulting problem size
grows exponentially with the dimension of the state-space.
This is commonly referred to as the
\textit{curse of dimensionality}.  The approach is particularly attractive
for feedback control of nonlinear systems with  lower dimensional
state space.  The method proposed in this paper also suffers from
the same drawback.
%We begin by reviewing some of the literature on HJB for continuous time systems.
Among the vast literature available on the topic of solving  the HJB equation, we briefly review some of the related literature.

Vinter \cite{vinterpaper} was the first to propose a linear programming approach for nonlinear optimal control of continuous time systems.  This was exploited to develop a numerical algorithm, based on semidefinite programming and density function-based formulation by Rantzer and co-workers in \cite{hedlundrantzer,prajnarantzer}.  Global almost everywhere stability of stochastic systems was studied by van Handel \cite{vanhandel}, using the density function.  Lasserre, Hern\'{a}ndez-Lerma, and co-workers \cite{LassHernBook,approxInfLP} formulated the control of Markov processes as a solution of the HJB equation. An adaptive space discretization approach
is used in \cite{Grune_04}; a cell mapping approach is used in
\cite{cell-cell1} and  \cite{Junge_Osinga, Junge_scl_05} utilizes set
oriented numerical methods to convert the HJB to one of finding
the minimum cost path on a graph derived from transition.
In
\cite{Hernandez_ocp, Gait_ocp, Lasserre_ocp}, solutions to stochastic
and deterministic optimal control problems are proposed, using
a linear programming approach or using a sequence of LMI relaxations. Our paper also draws some connection to research on optimization and
stabilization of controlled Markov chains discussed in \cite{Meyn_sadhana}. Computational techniques based on the viscosity solution of the HJB equation is proposed for the approximation of value function and optimal controls in \cite{viscosity_solnHJB_book} (Chapter VI).

Our proposed method, in particular the computational approach, draws some similarity with the above discussed references on the approximation of the solution of the HJB equation \cite{cell-cell1, Junge_Osinga,Junge_scl_05,viscosity_solnHJB_book}. Our method, too, relies on discretization of state space to obtain globally optimal stabilizing  control. However, our proposed approach differs from the above references in the following two fundamental ways. The first main difference arises due to adoption of non-classical weaker set-theoretic notion of almost everywhere stability for optimal stabilization. This weaker notion of stability allows for the existence of unstable dynamics in the complement of the stabilized attractor set; whereas, such unstable dynamics are not allowed using the classical notion of Lyapunov stability adopted in the above references. This weaker notion of stability is advantageous from the point of view of feedback control design. The notion of almost everywhere stability and density function for its verification was introduced by Rantzer in \cite{Rantzer01}. Furthermore, Rantzer proved, that unlike the control Lyapunov function, the co-design problem of jointly finding the density function and the stabilizing controller is convex \cite{Praly_conconvex}. The Lyapunov measure  used in this paper for optimal stabilization can be viewed as a measure corresponding to the  density function \cite{VaidyaMehtaTAC, Rajeev_continuous_time_journal}. Hence, it enjoys the same convexity property for the controller design. This convexity property, combined with the proposed linear transfer operator framework, is precisely exploited in the development of linear programming-based computational framework for optimal stabilization using Lyapunov measures. The second main difference compared to references \cite{Meyn_sadhana} and \cite{viscosity_solnHJB_book} is in the use of the discount factor $\gamma>1$ in the cost function (refer to Remark \ref{remark_discount}). The discount factor plays an important role in controlling the effect of finite dimensional discretization or the approximation process on the true solution. In particular, by allowing for the discount factor, $\gamma$, to be greater than one, it is possible to ensure that the control obtained using the finite dimensional approximation is {\it truly} stabilizing the nonlinear system \cite{Vaidya_CLM,arvind_ocp_online}.
%For more details on the comparison of our proposed approach with optimal control via approximate solutions of HJB equation, we refer the readers to \cite{Vaidya_CLM} (Section IV E) involving second author.

In a previous work \cite{Vaidya_CLM} involving Vaidya, the
problem of designing deterministic feedback controllers for
stabilization via control Lyapunov measure was addressed.  The authors
proposed solving the problem by using a mixed integer formulation or a non-convex
nonlinear program, which are not computationally efficient. There
are two main contributions of this paper. First, we show  a
deterministic stabilizing feedback controller can be constructed using a
computationally cheap tree-growing algorithm (Algorithm $1$,
Lemma \ref{stabGraphApp}).  The second main contribution of this paper is
 the extension of the Lyapunov measure framework introduced in \cite{VaidyaMehtaTAC}
 to design  optimal stabilization of an attractor set.
We prove  the optimal stabilizing controllers can be obtained
as the solution to a linear program. Unlike the approach proposed
in \cite{Vaidya_CLM}, the solution to the linear program is guaranteed
to yield deterministic controls. This paper is an extended version of the paper that appeared in the 2008 American Control Conference \cite{arvind_ocp}.

This paper is organized as follows. In Section \ref{prelim}, we
provide a brief overview of key results from
\cite{VaidyaMehtaTAC} and \cite{Vaidya_CLM} for stability
analysis, and stabilization of nonlinear systems using the Lyapunov
measure. The transfer
operators-based framework is used to formulate the OCP as an  infinite dimensional
linear program in Section \ref{optimal_control}. A computational
approach, based on set-oriented numerical methods, is proposed for
the finite dimensional approximation of the linear program in
Section \ref{computation}. %The effect of discretization on the optimal solution is discussed in Section \ref{section_finite_infinite} followed by
Simulation results are presented in Section \ref{examples}, followed by  conclusions in Section \ref{conclusion}.

\section{Lyapunov measure, stability and stabilization}\label{prelim}
The {\it Lyapunov measure} and {\it control Lyapunov measure}
were introduced in \cite{VaidyaMehtaTAC,Vaidya_CLM} for stability analysis and stabilizing controller design in discrete-time dynamical systems of the form,
\begin{equation}
x_{n+1}=F(x_n),\label{system}
\end{equation}
where $F: X\rightarrow X$ is assumed to be continuous with $X\subset \mathbb{R}^q$, a compact set. We denote by ${\cal B}(X)$ the Borel-$\sigma$ algebra on $X$ and ${\cal M}(X)$, the vector space of a real valued measure on ${\cal B}(X)$. The mapping, $F$, is assumed to be nonsingular with respect to the Lebesgue measure $\ell$, i.e., $\ell(F^{-1}(B))=0$, for all sets $B\in {\cal B}(X)$, such that $\ell(B)=0$. In this paper, we are interested in optimal stabilization of an attractor set defined as follows:

\begin{definition}[Attractor set]  A set ${\cal A}\subset X$ is said to be forward invariant under $F$, if $F({\cal A})={\cal A}$. A
closed forward invariant set, $\cal A$, is said to be an attractor set, if it there exists a neighborhood $V\subset X$ of $\cal A$, such that $\omega(x)\subset {\cal A}$ for all $x\in V$, where $\omega(x)$ is the $\omega$ limit set of $x$ \cite{VaidyaMehtaTAC}.
\end{definition}

 \begin{remark}\label{remark_nbd}In the following definitions and theorems, we will use the notation, $U(\epsilon)$, to denote the $\epsilon>0$ neighborhood of the attractor set $\cal A$ and $m\in {\cal M}(X)$, a finite measure absolutely continuous with respect to Lebesgue.\end{remark}

\begin{definition}[Almost everywhere stable with geometric decay]\label{stable_a.e.}
The attractor set ${\cal A}\subset X$ for a dynamical system (\ref{system}) is said to be almost everywhere (a.e.) stable with geometric decay with respect to some finite measure, $m\in {\cal M}(X)$, if given any $\epsilon>0$, there exists $M(\epsilon)<\infty$ and $\beta<1$, such that
$m\{x\in {\cal A}^c : F^{n}(x)\in X\setminus U(\epsilon)\}<M(\epsilon) \beta^n$.
\end{definition}

The above set-theoretic notion of a.e. stability is introduced in \cite{VaidyaMehtaTAC} and verified by using the
linear transfer operator framework. For the discrete time dynamical system (\ref{system}), the linear transfer Perron Frobenius (P-F) operator \cite{Lasota}
 denoted by $\mathbb{P}_F: {\cal M}(X)\rightarrow {\cal M}(X)$ is given by,

\begin{equation}
[{\mathbb P}_F \mu](B)=\int_{X}\chi_{B}(F(x))d\mu(x)=\mu(F^{-1}(B)),
\end{equation}
where $\chi_B(x)$ is the indicator function supported on the set
$B\in {\cal B}(X)$ and $F^{-1}(B)$ is the inverse image of set $B$. We define a sub-stochastic operator as a
restriction of the P-F operator on the complement of the attractor
set as follows:
\begin{equation}
 [{\mathbb P}^1_F\mu](B):=\int_{{\cal A}^c}\chi_{B}(F(x))d\mu(x)\label{rest_PF},
\end{equation}
for any set $B\in {\cal B}({\cal A}^c)$ and $\mu\in {\cal M}({\cal A}^c)$.
The condition for the a.e. stability of an attractor set
$\cal A$ with respect to some finite measure $m$ is defined in terms of
the existence of the {\it Lyapunov measure} $\bar \mu$, defined as follows \cite{VaidyaMehtaTAC}
\begin{definition}[Lyapunov measure]\label{def_Lya_measure}
The Lyapunov measure is defined as any non-negative measure $\bar
\mu$, finite outside
$U(\epsilon)$ (see Remark \ref{remark_nbd}),  and satisfies the following
inequality, $[{\mathbb P}_F^1 \bar \mu ](B)<\gamma^{-1} \bar \mu(B)$, for some $\gamma \geq1$ and all sets $B\in {\cal B}( X\setminus
U(\epsilon))$, such that $m(B)>0$.
\end{definition}

The following theorem from \cite{Vaidya_converse} provides the
condition for a.e. stability with geometric decay.
\begin{theorem} \label{theorem_converse} An attractor set $\cal A$
for the dynamical system (\ref{system}) is a.e. stable with
geometric decay with respect to finite measure $m$, if and only if for all $\epsilon>0$ there exists
a non-negative measure $\bar \mu$, which is finite on
${\cal B}(X\setminus U(\epsilon))$  and satisfies
%and a $\gamma>1$  such that following equality is satisfied

\begin{equation}
\gamma[ {\mathbb P}_F^1\bar \mu](B)-\bar \mu(B)=-m(B)\label{LME}
\end{equation} for all measurable sets
$B\subset X\setminus U(\epsilon)$ and for some $\gamma>1$.
\end{theorem}
\begin{proof}
We refer readers to Theorem \ref{theorem_converse} from \cite{arvind_ocp_online} for the proof.

\end{proof}

We consider the stabilization of dynamical systems of the form $x_{n+1}=T(x_n,u_n)$,
where $x_n\in X\subset \mathbb{R}^q$ and $u_n \in U\subset
\mathbb{R}^d$ are the state and the control input,
respectively. Both $X$ and $U$ are assumed  compact. The
objective is to design a feedback controller, $u_n=K(x_n)$, to
stabilize the attractor set $\cal A$. The stabilization problem is
solved using the Lyapunov measure by extending the P-F operator
formalism to the control dynamical system \cite{Vaidya_CLM}. We define
the feedback control mapping $C: X\rightarrow Y:= X\times U$ as
$C(x)=(x,K(x))$.
We denote by ${\cal B}(Y)$ the Borel-$\sigma$ algebra on $Y$ and ${\cal M}(Y)$ the vector space of real valued measures on ${\cal B}(Y)$. For any $\mu \in {\cal M}(X)$, the control mapping $C$  can be used to define a measure, $\theta\in {\cal M}(Y)$, as follows:

\begin{eqnarray}
&\theta(D):=[\mathbb P_C \mu](D)=\mu (C^{-1}(D))\nonumber\\
&[\mathbb P_{C^{-1}} \theta](B):=\mu(B)=\theta (C(B))\label{ess},
\end{eqnarray}
for all sets $D\in {\cal B}(Y)$ and $B\in {\cal B}(X)$. Since $C$ is an injective function with $\theta$ satisfying (\ref{ess}),  it follows from the theorem on disintegration of measure \cite{disintegration} (Theorem 5.8)  there exists a unique disintegration $\theta_x$ of the measure $\theta$  for $\mu$ almost all $x\in X$, such that $
\int_Y f(y)d\theta(y)=\int_X\int_{C(x)} f(y)d\theta_x(y)d\mu(x)$,
for any Borel-measurable function $f:Y\to \mathbb R$. In particular, for $f(y)=\chi_D(y)$, the indicator function for the set $D$, we obtain
$\theta(D)=\int_X \int_{C(x)}\chi_D(y) d\theta_x(y) d\mu(x)=[\mathbb{P}_C\mu](D).$
Using the definition of the feedback controller
mapping $C$, we write the feedback control system as $x_{n+1}=T(x_n, K(x_n))=T\circ C (x_n)$.
The system mapping $T: Y\rightarrow X$ can be associated with P-F operators ${\mathbb
P}_T: {\cal M}(Y)\rightarrow {\cal M}(X)$ as $
[{\mathbb P}_T \theta](B)=\int_{Y}  \chi_B(T(y))d\theta(y)$.
%The advantage of
%writing the feedback control dynamical system as the composition
%of two maps $T: Y\rightarrow X$ and $C: X\rightarrow Y$ is that
The P-F operator for the composition $T\circ C: X\rightarrow X$
can be written as a product of ${\mathbb P}_T$ and ${\mathbb P}_C$. In particular, we obtain \cite{arvind_ocp_online}
\begin{eqnarray*}
&&[{\mathbb P}_{T\circ C} \mu](B)=\int_Y \chi_B(T(y))d [\mathbb{P}_C \mu](y)\nonumber\\&=&[\mathbb{P}_T \mathbb{P}_C \mu](B)=\int_X \int_{C(x)} \chi_{B}(T(y))d\theta_x(y)d\mu(x)\label{PTC}.
\end{eqnarray*}

The P-F operators, $\mathbb{P}_T$ and $\mathbb{P_C}$, are used to define their restriction, $\mathbb{P}_T^1:{\cal M}({\cal A}^c\times U)\to {\cal M}({\cal A}^c)$, and $\mathbb{P}_C^1:{\cal M}({\cal A}^c)\to {\cal M}({\cal A}^c\times U)$ to the complement of the attractor set, respectively, in a way similar to  Eq. (\ref{rest_PF}).
The control Lyapunov measure introduced
in \cite{Vaidya_CLM} is defined as any non-negative
measure $\bar \mu\in {\cal M}({\cal A}^c)$,  finite on ${\cal
B}(X\setminus U(\epsilon))$, such that there exists a control mapping $C$ that satisfies $[{\mathbb P}^1_T  {\mathbb P}^1_C \bar \mu ](B)<\beta \bar \mu(B)$,
for every set $B\in {\cal B}( X\setminus U(\epsilon))$ and $\beta\leq 1$.
Stabilization of the  attractor set is posed
as a co-design problem of jointly obtaining the control Lyapunov
measure $\bar \mu$ and the control P-F operator ${\mathbb P}_C$ or
in particular disintegration of measure $\theta$, i.e., $\theta_x$. The disintegration measure $\theta_x$, which lives on the fiber of $C(x)$, in general, will not be absolutely continuous with respect to Lebesgue. For the deterministic control map, $K(x)$, the conditional measure, $\theta_x(u)=\delta(u-K(x))$, the Dirac delta measure. However, for the purpose of computation, we relax this  condition.
The purpose of this paper and the following sections are to extend
the Lyapunov measure-based framework for the optimal stabilization
of nonlinear systems. One of the key highlights of this paper is
 the {\it deterministic} finite optimal stabilizing control is
obtained as the solution for a finite linear program.
%The finite
%deterministic stabilizing control is obtained as a special case of
%optimal stabilizing control, when the cost function is constant and
%independent of the states and the control.

\section{Optimal stabilization}\label{optimal_control}
The objective is to design a feedback controller for the
stabilization of the attractor set, $\cal A$, in a.e. sense, while
minimizing a suitable cost function. Consider the following
control system,
\begin{equation} x_{n+1}=T(x_n,u_n),
\label{control_syst}
\end{equation}
where $x_n\in X\subset \mathbb{R}^q$ and $u_n\in U\subset
\mathbb{R}^d$ are state and control input, respectively, and $T:X\times U=:Y\to X$. Both $X$ and $U$ are assumed compact. We define $X_1:=X\setminus U(\epsilon)$.
%, where
%$U(\epsilon)$ is the $\epsilon$ neighborhood of the attractor set
%$\cal A$.
\begin{assumption}
We assume  there exists a feedback controller mapping $C_0(x)=(x,K_0(x))$, which locally stabilizes the
invariant set ${\cal A}$, i.e., there exists a neighborhood $V$ of $\cal A$
such that $T\circ C_0(V)\subset V$ and $x_n\rightarrow \cal A$ for all
$x_0\in V$; moreover ${\cal A} \subset U(\epsilon)\subset V$.
\end{assumption}

Our objective is to construct the optimal stabilizing controller
for almost every initial condition starting from $X_1$. Let
$C_{1}: X_1\rightarrow Y$ be the stabilizing control map for $X_1$.
The control mapping $C: X\to X\times U$ can be written as follows:
\begin{equation}C(x)=\left \{\begin{array}{ccl}C_0(x)=(x,K_0(x))&
{\rm for}&x\in U(\epsilon)\\
C_1(x)=(x,K_1(x))&{\rm for}&x\in X_1.
\end{array}\right.
\end{equation}

Furthermore, we assume the feedback control system $T\circ C:
X\rightarrow X$ is non-singular with respect to the Lebesgue measure,
$m$. We seek to design the controller mapping, $C(x)=(x,K(x))$, such that the attractor set $\cal A$ is a.e. stable with geometric decay rate $\beta<1$, while minimizing the following cost function,
\begin{equation} {\cal
C}_C(B)=\int_{B}\sum_{n=0}^{\infty} \gamma^n G\circ C(x_n)
dm(x),\;\;\;\;\label{cost_sum}
\end{equation}
where $x_0=x$, the cost function $G: Y\to {\mathbb R}$ is assumed a continuous non-negative real-valued function,  such that $G({\cal A},0)=0$, $x_{n+1}=T\circ C(x_n)$, and $0<\gamma<\frac{1}{\beta}$. Note, that in the cost function (\ref{cost_sum}), $\gamma$ is allowed  greater than one and this is one of the main departures from the conventional optimal control problem, where $\gamma\leq 1$. However, under the assumption that the controller mapping $C$ renders the attractor set a.e. stable with a geometric decay rate, $\beta<\frac{1}{\gamma}$, the cost function (\ref{cost_sum}) is finite.

 \begin{remark}
To simplify the notation, in the following we will use the notion of the scalar product between continuous function $h\in {\cal C}^0(X)$ and measure $\mu\in {\cal M}(X)$ as $\left<h,\mu\right>_X :=\int_X h(x)d\mu(x)$ \cite{Lasota}.
% In this notation  the weak convergence of measures has a simple form. Namely , $\{\mu_n\}$ converges to $\mu$ weakly if $\lim_{n\to \infty}\left< h,\mu_n\right>=\left<h,\mu\right>,\;\;\;\;\;\forall h\in C^0(X)$.
 \end{remark}
 The following theorem proves  the cost of stabilization of the
 set $\cal A$ as given in Eq. (\ref{cost_sum}) can be expressed using the control Lyapunov measure equation.
\begin{theorem}\label{theorem_ocp} Let the controller mapping, $C(x)=(x,K(x))$, be such that the attractor set $\cal A$ for the feedback control system $T\circ C:X\to X$ is a.e. stable with geometric decay rate $\beta<1$. Then, the cost function (\ref{cost_sum}) is well defined for $\gamma<\frac{1}{\beta}$ and, furthermore, the cost of stabilization of the attractor set $\cal A$ with respect to Lebesgue almost every initial condition starting from set $B\in {\cal B}(X_1)$ can be expressed as follows:

\begin{eqnarray}
&{\cal C}_C(B)= \int_{B}\sum_{n=0}^{\infty}\gamma^n G\circ C(x_n)
dm(x)\nonumber\\&=\int_{{\cal A}^c\times U}G(y)d[\mathbb{P}_C^1 \bar \mu_B](y)=
\left<G,{\mathbb P}_C^1 \bar \mu_B\right>_{{\cal A}^c\times
U},\label{inequality}
\end{eqnarray}
where, $x_0=x$ and $\bar \mu_B$ is the solution of the following  control
Lyapunov measure equation,
\begin{eqnarray}
\gamma \mathbb{P}_T^1\cdot \mathbb{P}_C^1\bar \mu_B(D)-\bar
\mu_B(D)=-m_B(D), \label{control_Ly}
\end{eqnarray}
for all $D\in {\cal B}(X_1)$ and  where $m_B(\cdot):=m(B\cap \cdot)$ is a finite measure
supported on the set $B\in {\cal B}(X_1)$.
\end{theorem}
\begin{proof}
The proof is omitted in this paper, due to limited space but can be found in the online version of the paper \cite{arvind_ocp_online} (Theorem \ref{theorem_ocp}).
\end{proof}
 By appropriately selecting the measure on the right-hand side of the control Lyapunov measure equation (\ref{control_Ly}) (i.e., $m_B$),
stabilization of the attractor set with respect to
a.e. initial conditions starting from a particular set can
be studied.
The minimum cost of stabilization is defined as the minimum over
all a.e. stabilizing controller mappings, $C$, with a geometric decay as follows:
\begin{equation}
{\cal C}^{*}(B)=\min_{C}{\cal C}_C(B).
\end{equation}
Next, we write the infinite dimensional linear program for the optimal stabilization of the attractor
set $\cal A$. Towards this goal, we first define the projection map,
$P_1: {\cal A}^c\times U\rightarrow {\cal A}^c$
as:
$P_1(x,u)=x,$
and denote the P-F operator corresponding to $P_1$ as
$\mathbb{P}_{P_1}:  {\cal M}({\cal A}^c\times U)\rightarrow
{\cal M}({\cal A}^c)$, which can be written as
$[{\mathbb P}^1_{P_1} \theta](D)=\int_{{\cal A}^c\times U}\chi_D(P_1(y))d\theta(y)=\int_{D\times U}d\theta(y)=\mu(D)$.
 Using this definition of projection mapping, $P_1$, and the
corresponding P-F operator, we can write the linear program for the optimal stabilization of set $B$
with  unknown variable
$\theta$ as follows:
\begin{equation}
\min\limits_{\theta\geq 0} \left<G, \theta\right>_{{\cal A}^c \times U},\;\;
\mbox{s.t. } \gamma [{\mathbb P}^1_T \theta](D)-[{\mathbb P}^1_{P_1}
\theta](D)=-m_B(D)\label{linear_program},
\end{equation}
for $D\in {\cal B}(X_1)$.
\begin{remark}\label{remark_discount}
Observe  the geometric decay parameter satisfies $\gamma > 1$.
This is in contrast to most optimization problems studied in
the context of Markov-controlled processes, such as in Lasserre and
Hern\'{a}ndez-Lerma \cite{LassHernBook}. Average cost and
discounted cost optimality problems are considered in
\cite{LassHernBook, viscosity_solnHJB_book}. The additional flexibility provided by $\gamma>1$ guarantees the controller obtained from the finite dimensional approximation of the infinite dimensional program (\ref{linear_program}) also stabilizes the attractor set for system (\ref{control_syst}). For a more detailed discussion on the role of $\gamma$ on the finite dimensional approximation, we refer  readers to the online version of the paper \cite{arvind_ocp_online}.
\end{remark}

%%%%%%%%%%%%%%%%%%%%%%%%%%%%%%%%%%%%%%%%%%%%%%%%%%%%%%%%%%%%%%%%%%%%%%%%%%%%%%%%%%%%

\section{Computational approach}\label{computation}
The objective of the present section is to present a computational
framework for the solution of the finite-dimensional approximation of
the optimal stabilization problem in \eqref{linear_program}.
There exists a number of references related to the solution of infinite
dimensional linear programs (LPs), in general, and those arising from the control
of Markov processes. Some will be described next.
The monograph by Anderson and Nash \cite{AndNash} is an excellent
reference on the properties of infinite dimensional LPs.

Our intent is to use the finite-dimensional approximation as a tool to obtain
 stabilizing controls to the infinite-dimensional system. First, we will derive conditions under which solutions to the finite-dimensional approximation exist.
 %This result will be used to  show existence of a deterministic controller.

%\subsection{Problem set-up in finite dimensions}

Following
\cite{VaidyaMehtaTAC} and \cite{Vaidya_CLM}, we discretize the state-space and
control space  for the purposes of computations as described below.
Borrowing the notation from \cite{Vaidya_CLM}, let
${\cal X}_N := \{D_1,...,D_i,...,D_N\}$ denote a finite
partition of the state-space $X \subset \R^q$.
The measure space associated with ${\cal X}_N$ is $\R^N$.
We assume without loss of generality that the
attractor set, ${\cal A}$, is contained in $D_{N}$, that is,
${\cal A} \subseteq D_N$.
Similarly, the control space, $U$, is quantized and the control input
is assumed to take only finitely many control values from the
quantized set,
${\cal U}_M = \{u^1,\hdots,u^a,\hdots,u^M\}$,
where $u^a \in \R^d$.  The partition, ${\cal U}_M$, is identified with the vector
space, $\R^{N \times M}$. The system map that results from choosing the controls $u_N$ is denoted as
$T_{u_N}$ and the corresponding P-F operator is denoted as
$P_{T_{u_N}} \in \R^{N \times N}$.
Fixing the controls on all sets of
the partition to $u^a$, i.e., $u_N(D_i) = u^a$, for all $D_i \in
{\cal X}_N$, the system map that results is denoted as
$T_a$ with the corresponding P-F operator  denoted as
$P_{T_a} \in \R^{N \times N}$.  The entries for $P_{T_a}$ are calculated as:
$
\sub{P_{T_a}}{ij} := \frac{m({T_a}^{-1}(D_j)\cap D_i)}{m(D_i)}
$, where $m$ is the Lebesgue measure and $\sub{P_{T_a}}{ij}$ denotes the
$(i,j)$-th entry of the matrix.
Since $T_a:X \rightarrow X$, we have  $P_{T_a}$ is a Markov matrix.
Additionally, $\oP_{T_a} : \R^{N-1} \rightarrow \R^{N-1}$
will denote the finite dimensional counterpart of the
P-F operator restricted to ${\cal X}_N \setminus D_N$,
the complement of the attractor set.
It is easily seen that $\oP_{T_a}$ consists of the first
$(N-1)$ rows and columns of $P_{T_a}$.

In \cite{VaidyaMehtaTAC} and \cite{Vaidya_CLM}, stability analysis and
stabilization of the attractor set are studied,
using the above finite dimensional approximation of the P-F
operator. The finite dimensional approximation of the P-F operator results in  a weaker notion of stability, referred to as coarse stability \cite{VaidyaMehtaTAC,arvind_ocp_online}. Roughly speaking, coarse stability means stability modulo
attractor sets with domain of attraction smaller than
the size of cells within the partition.

With the above quantization of the control space and partition of the
state space, the determination of the control $u(x) \in U$ (or
equivalently $K(x)$) for all $x \in {\cal A}^c$  has now been cast as a problem
of choosing $u_N(D_i) \in {\cal U}_M$ for all sets $D_i
\subset {\cal X}_N$.  The finite dimensional approximation of the
optimal stabilization problem \eqref{linear_program} is equivalent to
solving the following finite-dimensional LP:
\begin{equation}
 \min\limits_{\theta^a,\mu \geq 0} \mbox{ }
\sum_{a=1}^M (G^a)^{'} \theta^a,\;\;\;
 \mbox{s.t. } \gamma\sum_{a=1}^{M} (P_{T_a})^{'}\theta^a
 - \sum_{a=1}^M\theta^a = -m
\label{lyaMeasLP},
\end{equation}
where we have used the notation $(\cdot)^{'}$ for the transpose operation,
$m \in \R^{N-1}$ and $\sub{m}{j} > 0$ denote the support of Lebesgue measure,
$m$, on the set $D_j$, $G^a \in \R^{N-1}$
is the cost defined on ${\cal X}_N\setminus D_N$ with
$\sub{G^a}{j}$ the cost associated with using control action $u^a$
on set $D_j$; $\theta^a \in \R^{N-1}$ are, respectively, the
discrete counter-parts of infinite-dimensional measure quantities in
\eqref{linear_program}. In the LP \eqref{lyaMeasLP}, we have not enforced the constraint,
\begin{equation}
\sub{\theta^a}{j} > 0 \mbox{ for exactly one } a \in \{1,...,M\},
\label{detControl}
\end{equation}
for each $j = 1,...,(N-1)$.  The above constraint ensures  the
control on each set in unique.  We prove in the following  the
uniqueness can be ensured without  enforcing the constraint,
provided the LP \eqref{lyaMeasLP} has a solution.
To this end, we introduce the dual LP
associated with the LP in \eqref{lyaMeasLP}.
The dual to the LP in \eqref{lyaMeasLP} is,
\vspace{-0.1in}
\begin{equation}
\max\limits_{V} \mbox{ } m^{'}V ,\;\;\;
\mbox{s.t. } V \leq \gamma\oP_{T_a}V + G^a \mbox{ } \forall
a = 1,...,M.
\label{lyaLP}
\end{equation}
In the above LP \eqref{lyaLP}, $V$ is the dual variable to the
equality constraints in \eqref{lyaMeasLP}.
\vspace{-0.1in}
\subsection{Existence of solutions to the finite LP}

We make the following assumption throughout this section.
\begin{assumption} \label{assumption_A}
There exists $\theta^{a} \in \R^{N-1} \;\forall\; a = 1,\ldots,M$,
such that the LP in \eqref{lyaMeasLP} is feasible for
some $\gamma > 1$.
\end{assumption}
Note, Assumption \ref{assumption_A} does not impose the requirement in
\eqref{detControl}.   For the sake of simplicity and clarity of
presentation, we will assume that the measure, $m$, in
\eqref{lyaMeasLP} is equivalent to the Lebesgue measure and $G > 0$.
Satisfaction of Assumption \ref{assumption_A} can be
verified using the following algorithm. \newline
\textbf{Algorithm $1$}
1) Set ${\cal I} := {1, ..., N-1}$, ${\cal I}_0 := {N}$, $L = 0$.\;\;\;2) Set ${\cal I}_{L+1} := \emptyset$.\;\;\;3) For each $i \in {\cal I} \setminus \{{\cal I}_0 \cup \ldots \cup {\cal
I}_L \}$ do\;\;\;a) Pick the smallest $a  \in {1, ..., M }$ such that
$\sub{P_{T_a}}{ij} > 0$ for some $j \in {\cal I}_L$.\;\;\;b) If $a$ exists then, set $u_N(D_i) := u^a$, ${\cal I}_{L+1} := {\cal
I}_{L+1} \cup \{i\}$.\;\;\;
4) End For\;\;\;
5) If ${\cal I}_0 \cup \ldots \cup {\cal I}_L = {\cal I}$ then,
set $L = L + 1$. STOP.\;\;\;
6) If ${\cal I}_{L+1} = \emptyset$ then, STOP.\;\;\;
7) Set $L = L + 1$. Go to Step 2.

The algorithm iteratively adds to ${\cal I}_{L+1}$, set $D_i$,
which has a non-zero probability of transition to any of the sets in
${\cal I}_L$.  In graph theory terms, the above algorithm iteratively
builds a tree starting with the set $D_N \supseteq {\cal A}$.
If the algorithm terminates in Step $6$,  then we
have identified sets
${\cal I} \setminus \{{\cal I}_0 \cup \ldots \cup {\cal I}_{L}\}$
that cannot be stabilized with the controls in ${\cal U}_M$.
If the algorithm terminates at Step $5$, then
we show in the Lemma below that a set of stabilizing controls exist.

\begin{lemma}\label{stabGraphApp}
Let ${\cal X}_N = \{D_1,\ldots,D_n\}$ be a partition of the
state space, $X$, and  ${\cal U}_M = \{u^1,\ldots,u^M\}$
be a quantization of the control space, $U$.  Suppose  Algorithm $1$
terminates in Step $5$ after $L^{\mathrm{max}}$ iterations, then the controls
$u_N$ identified  by the algorithm renders the system coarse stable.
\end{lemma}
\begin{proof}
Let $P_{T_{u_N}}$ represent the closed loop transition matrix
resulting from the controls identified by Algorithm 1.
Suppose $\mu \in \R^{N-1}$, $\mu \geq 0, \mu \neq 0$ be any initial distribution
supported on the complement of the attractor set ${\cal X}_N \setminus
D_N$.   By construction, $\mu$ has a non-zero probability of entering the attractor set after
$L^{\mathrm{max}}$ transitions.  Hence,
\begin{equation*}
\sum\limits_{i=1}^{N-1}\sub{\mu'(P^1_{T_{u_N}})^{L^{\mathrm{max}}}}{i} <
\sum\limits_{i=1}^{N-1}\sub{\mu}{i}
\implies
\lim_{n \rightarrow \infty}(P^1_{T_{u_N}})^{nL^{\mathrm{max}}} \longrightarrow 0.
\end{equation*}
Thus, the sub-Markov matrix $P^1_{T_{u_N}}$ is transient and implies the claim.
\end{proof}
%\begin{proof}
%We refer the readers to Lemma \ref{stabGraphApp} from  \cite{arvind_ocp_online} for the proof.
%\end{proof}

Algorithm 1 is  less expensive than the approaches proposed in \cite{Vaidya_CLM}, where a
mixed integer LP and a nonlinear programming approach
were proposed. The strength of our algorithm is that it is guaranteed to find
deterministic stabilizing controls, if they exist. The following lemma shows an optimal solution to \eqref{lyaMeasLP} exists under Assumption \ref{assumption_A}.

\begin{lemma}\label{strongDuality}
Consider a partition ${\cal X}_N = \{D_1,\ldots,D_N\}$ of the state-space
$X$ with attractor set ${\cal A} \subseteq D_N$ and a quantization
 ${\cal U}_M = \{u^1,\ldots,u^M\}$ of the control space $U$.  Suppose
 Assumption \ref{assumption_A} holds for some $\gamma > 1$ and for $m, G > 0$.
Then, there exists an
optimal solution, $\theta$, to the LP \eqref{lyaMeasLP} and
an optimal solution, $V$, to the dual LP \eqref{lyaLP}
with equal objective values, ($\sum\limits_{a=1}^M(G^a)^{'}\theta^a
= m'V$) and $\theta, V$ bounded.
\end{lemma}
\begin{proof}
From Assumption \ref{assumption_A}, the LP \eqref{lyaMeasLP} is feasible.
Observe  the dual LP in \eqref{lyaLP} is always feasible
with a choice of $V = 0$.  The feasibility of primal and dual linear
programs implies the claim as a result of LP strong
duality \cite{MR1297120}.
\end{proof}

\begin{remark}
Note, existence of an optimal solution does not impose a positivity
requirement on the cost function, $G$.
In fact, even assigning $G = 0$ allows determination of a
stabilizing control from the Lyapunov measure equation
\eqref{lyaMeasLP}.  In this case, any feasible solution to \eqref{lyaMeasLP} suffices.
\end{remark}

The next result shows the LP \eqref{lyaMeasLP} always admits an optimal solution
satisfying \eqref{detControl}.

\begin{lemma} \label{compLemma1}
Given a partition ${\cal X}_N = \{D_1,\ldots,D_N\}$ of the state-space,
$X$, with attractor set, ${\cal A} \subseteq D_N$, and a quantization,
${\cal U}_M = \{u^1,\ldots,u^M\}$, of the control space, $U$.  Suppose
Assumption \ref{assumption_A} holds for some $\gamma > 1$ and for $m, G > 0$.  Then, there
exists a solution $\theta \in \R^{N-1}$ solving \eqref{lyaMeasLP}
and $V \in \R^{N-1}$ solving \eqref{lyaLP} for any
$\gamma \in [1,\overline{\gamma}_N)$.  Further, the following hold at
the solution: 1) For each $j = 1,...,(N-1)$, there exists at least one $a_j \in 1,...,M$,
  such that $\sub{V}{j} = \gamma\sub{\oP_{T_{a_j}}V}{j} + \sub{G^{a_j}}{j}$
  and $\sub{\theta^{a_j}}{j} >0$. 2) There exists a $\tilde{\theta}$ that solves \eqref{lyaMeasLP},  such that for each $j = 1,...,(N-1)$, there is exactly one
$a_j \in 1,...,M$, such that $\sub{\tilde{\theta}^{a_j}}{j} > 0$ and
$\sub{\tilde{\theta}^{a'}}{j} = 0$ for $a' \neq a_j$.
\end{lemma}
\begin{proof}

From the assumptions, we have that Lemma \ref{strongDuality} holds.
Hence, there exists $\theta \in \R^{N-1}$ solving \eqref{lyaMeasLP} and
$V \in \R^{N-1}$ solving \eqref{lyaLP} for any
$\gamma \in [1,\overline{\gamma}_N)$.  Further, $\theta$ and $V$
satisfy following the first-order optimality conditions \cite{MR1297120},
\begin{equation}
\begin{array}{c}
\sum\limits_{a=1}^M \theta^a -
\gamma\sum\limits_{a=1}^M(\oP_{T_a})'\theta^a = m \\
V \leq \gamma\oP_{T_a}V + G^a \perp \theta^a \geq 0 \mbox{ }\forall
a = 1,...,M.
\end{array} \label{statcond}
\end{equation}
We will prove each of the claims in order.

\emph{Claim 1:}
Suppose, there exists $j \in 1,...,(N-1)$, such that
$\sub{\theta^a}{j} = 0$ for all $a = 1,...,M$. Substituting in the
optimality conditions \eqref{statcond}, one obtains,
\begin{equation*}
\gamma\sum\limits_{a=1}^M\sub{(\oP_{T_a})'\theta^a}{j} = -\sub{m}{j}
\end{equation*}
which cannot hold, since, $\oP_{T_a}$ has non-negative entries, $\gamma > 0$ and
$\theta^a \geq 0$.   Hence, there exists at least one $a_j$ such that
$\sub{\theta^{a_j}}{j} > 0$.  The complementarity condition in
\eqref{statcond} then requires that $\sub{V}{j} =
\sub{\gamma\oP_{T_{a_j}}V}{j} + \sub{G^{a_j}}{j}$.  This proves the first claim.

\emph{Claim 2:} Denote $a(j) = \min \{a | \sub{\theta^a}{j} > 0\}$ for each
$j = 1,...,(N-1)$. The existence of $a(j)$ for each $j$ follows from
statement $1$. Define $\oP_{T_{u}} \in \R^{(N-1) \times (N-1)}$ and
$G^u \in \R^{N-1}$ as follows:
\begin{equation}
\begin{array}{rcl}
\sub{\oP_{T_u}}{ji} & := & \sub{\oP_{T_{a(j)}}}{ji} \mbox{ }
\forall i = 1,...,(N-1) \\
\sub{G^u}{j} & := & \sub{G^{a(j)}}{j}
\end{array}
\end{equation}
for all $j = 1,...,(N-1)$.  From the definition of $\oP_{T_{u}}$, $G^u$
and the complementarity condition in \eqref{statcond}, it is
easily seen
that $V$ satisfies
\begin{equation}
V = \gamma\oP_{T_u}V + G^u = \lim_{n\rightarrow
\infty}((\gamma\oP_{T_u})^n V+
\sum\limits_{k=0}^n(\gamma\oP_{T_u})^k G^u).\label{dpVeqns}
\end{equation}
Since $V$ is bounded and $G^u > 0$, it follows that $\rho(\oP_{T_u}) < 1/\gamma$.
Define $\tilde{\theta}$ as,
\begin{subequations}
%\begin{array}{c}
\begin{equation}
\left[ \begin{array}{c}
\sub{\tilde{\theta}^{a(1)}}{1} \\
\vdots \\ \sub{{\tilde{\theta}}^{a(N-1)}}{N-1} \end{array} \right] =
(I_{N-1} - \gamma(\oP_{T_u})^{'})^{-1}m \label{defNZtheta}
\end{equation}
\begin{equation}
\sub{\tilde{\theta}^a}{j} = 0 \mbox{ } \forall j = 1,...,(N-1), \quad a \neq
a(j). \label{defzerotheta}
\end{equation}
\end{subequations}
The above is well-defined, since we have already shown that
$\rho(\oP_{T_u}) < 1/\gamma$.

From the construction of $\tilde{\theta}$, we have that for each $j$
there exists only one $a_j$,  namely $a(j)$, for which
$\sub{\tilde{\theta}^{a(j)}}{j} > 0$. It remains to show that
$\tilde{\theta}$ solves \eqref{lyaMeasLP}.
For this, observe
\begin{equation}
\begin{array}{rcl}
\sum\limits_{a=1}^M(G^a)^{'}{\tilde{\theta}}^a
& \overset{\eqref{defzerotheta}}{=}
& \sum\limits_{j=1}^{N-1}\sub{G^{a(j)}_j{\tilde{\theta}}^{a(j)}}{j}
 \overset{\eqref{defNZtheta}}{=}
(G^u)^{'}(I_{N-1} - \gamma(\oP_{T_u})^{'})^{-1}m \\
& = & ((I_{N-1} - \gamma\oP_{T_u})^{-1}G^u)^{'}m
 \overset{\eqref{dpVeqns}}{=}  V^{'}m.
\end{array}
\end{equation}
The primal and dual objectives are equal with the above definition
of $\tilde{\theta}$.  Hence, $\tilde{\theta}$ solves
\eqref{lyaMeasLP}.  The claim is proved.
%\end{enumerate}
\end{proof}
%\begin{proof}
%We refer the readers to Lemma \ref{compLemma1} from \cite{arvind_ocp_online} for the proof.
%\end{proof}

The following theorem states the main result.

\begin{theorem}\label{thmComp}
Consider a partition ${\cal X}_N = \{D_1,\ldots,D_N\}$ of the state-space,
$X$, with attractor set, ${\cal A} \subseteq D_N$, and a quantization,
${\cal U}_M = \{u^1,\ldots,u^M\}$, of the control space, $U$.
Suppose Assumption \ref{assumption_A} holds for some $\gamma > 1$ and for
$m, G > 0$. Then, the following statements hold: 1) there exists a bounded $\theta$, a solution to
  \eqref{lyaMeasLP} and a bounded $V$, a solution
  to \eqref{lyaLP}; 2) the optimal control for each set, $j= 1,...,(N-1)$, is given by
$ u(D_j) = u^{a(j)},\;\;\mbox{where } a(j) := \min\{a | \sub{\theta^a}{j} > 0\}$; 3) $\mu$ satisfying
$\gamma(\oP_{T_u})^{'}\mu - \mu = -m\;\;, \mbox{ where }
\sub{\oP_{T_u}}{ji} = \sub{\oP_{T_{a(j)}}}{ji}$
is the Lyapunov measure for the controlled system.
\end{theorem}
\begin{proof}
Assumption \ref{assumption_A} ensures that the linear programs
\eqref{lyaMeasLP} and \eqref{lyaLP} have a finite optimal solution
(Lemma \eqref{strongDuality}).  This proves the first claim of the
theorem and also allows the applicability of Lemma \ref{compLemma1}.
The remaining claims follow as a consequence.
\end{proof}

Although the results in this section  assumed  the measure, $m$,
is equivalent to Lebesgue, this can be easily relaxed to the case
where $m$ is absolutely continuous with respect to Lebesgue and is of
interest where the system is not everywhere stabilizable. If it is
known there are regions of the
state-space  not stabilizable, then $m$ can be chosen such
that its support is zero on these regions.  If the regions are not
known a priori then, \eqref{lyaMeasLP} can be modified to minimize
the $l_1$-norm of the constraint residuals.  This is similar to
the \emph{feasibility phase}  commonly employed in LP algorithms \cite{PDIPM}.

%\section{Stability in finite dimension}
%\input{finite-infinite_v2}

\section{Examples}\label{examples}

The results in this section  have been obtained using an interior-point
algorithm, \textsf{IPOPT} \cite{ipopt}. \textsf{IPOPT} is an
open-source software available through the COIN-OR repository \cite{coinorrep}, developed for solving
large-scale non-convex nonlinear programs. We solved the linear programs using the Mehrotra predictor-corrector algorithm \cite{NocedalWright}  for linear and convex quadratic programs implemented in IPOPT.

\subsection{Standard Map}

The standard map is a 2D map also known as the Chirikov map. It is
one of the most widely studied maps in dynamical systems
\cite{Vaidya_control_twist}.

The standard map is obtained by taking a Poincare map of one degree of freedom Hamiltonian system forced with a time-periodic perturbation and by writing the system in {\it action-angle} coordinates. The standard map forms a basic building block for studying higher degree of
freedom Hamiltonian systems. The control standard map is described by the
following set of equations \cite{Vaidya_control_twist}:
\begin{eqnarray}
&x_{n+1} =x_{n} + y_{n} + Ku\sin{{2}{\pi}{x_{n}}} \;\;({\rm mod \;1})\nonumber\\
&y_{n+1} = y_{n} + Ku\sin{{2}{\pi}{x_{n}}} \;\;({\rm mod \;1}),
\end{eqnarray}
where $(x,y)\in R:=\{[0,1)\times
[0,1)\}$ and $u$ is the control input. The dynamics of the standard map with $u \equiv 1$ and $K=0.25$ is shown in Fig. \ref{standard_map}. The phase space dynamics consists of a mix of  chaotic and regular regions. With the increase in the perturbation parameter, $K$, the dynamics of the system become more chaotic.
\begin{figure}
\begin{center}
\vspace{-0.1in}
{\scalebox{0.35}{\includegraphics{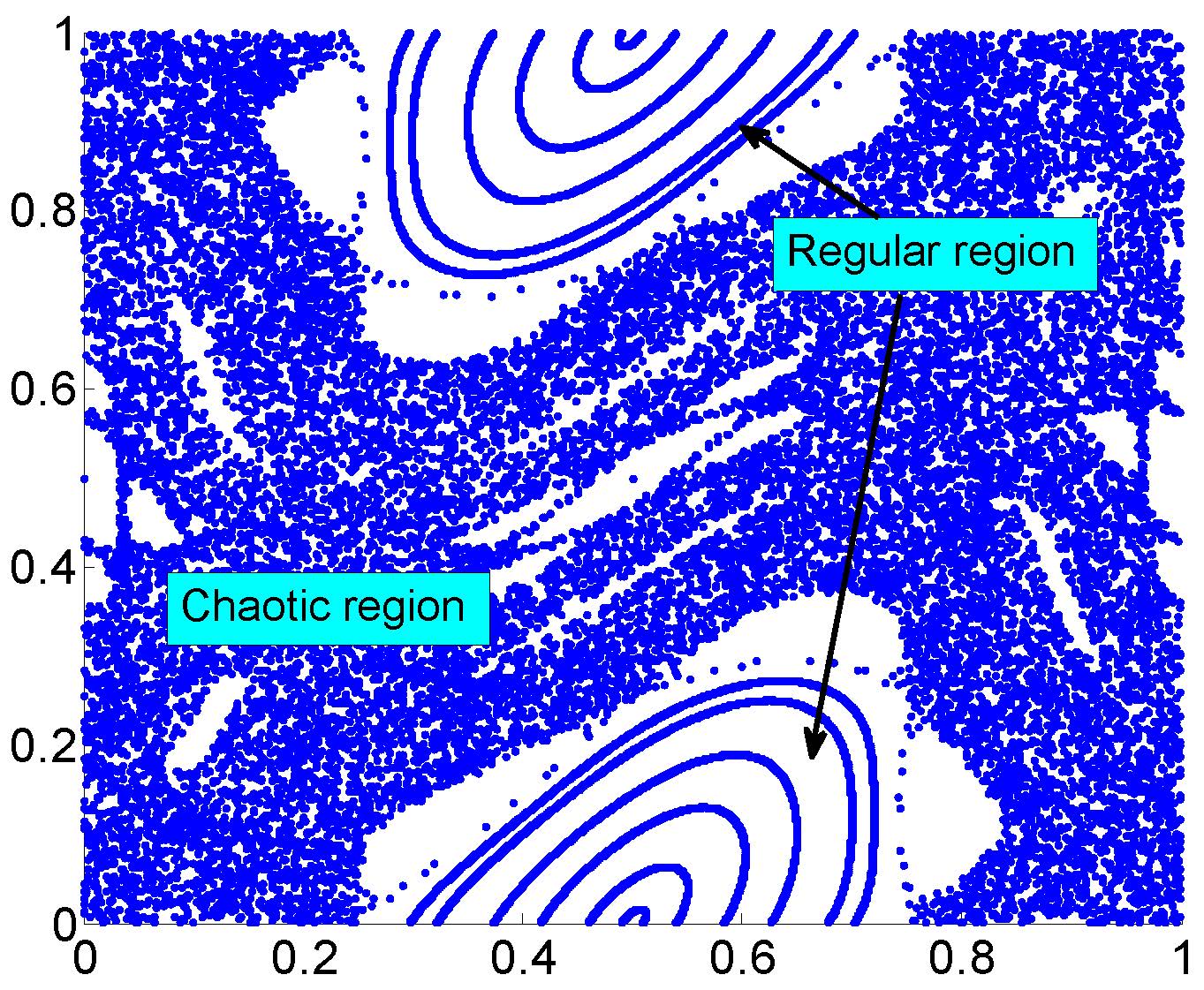}}}
\vspace{-0.1in}
\caption{Dynamics of the standard map.}
\label{standard_map}
\end{center}
\end{figure}

For the uncontrolled system (i.e., $u\equiv 0$), the entire phase space is foliated with periodic and quasi-periodic orbits. In particular, every $y=$ constant line is invariant under the system dynamics, and  consists of periodic and quasi-periodic orbits for rational and irrational values of $y$, respectively. Our objective is to globally stabilize the period two orbit located at $(x,y) =
(0.25,0.5)$ and $(x,y) = (0.75,0.5)$, while minimizing the cost
function $G(x,y,u)=x^2+y^2+u^2$. To construct the finite dimensional
approximation of the P-F operator, the phase space, $R$, is partitioned into $50\times 50$ squares: each cell has 10 sample points.
The discretization for the control space is chosen to be equal to $u^a = [-0.5 : 0.05 : 0.5]$. In Figs. \ref{stp2}a and \ref{stp2}b,  we show the plot for optimal cost function and control input, respectively. We observe in Fig. \ref{stp2}b  the control values used to control the system are approximately antisymmetric about the origin. This antisymmetry is inherent in the standard map and can also be observed in the uncontrolled standard map plot in Fig. \ref{stp2}b.

\begin{figure}[here] \centering
\subfigure[]
{\includegraphics[width=2.5in]{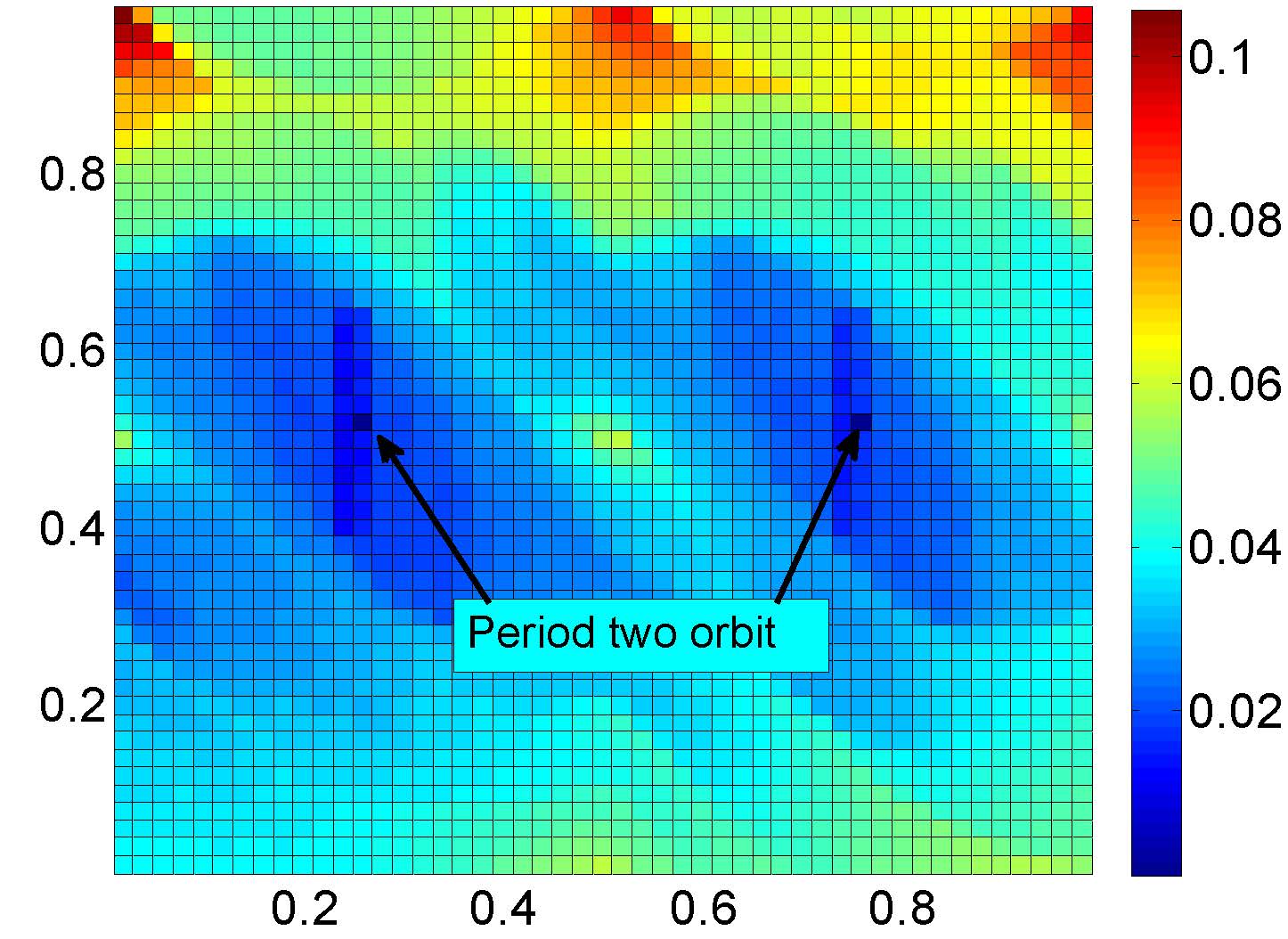}} \hspace{0.05in}\subfigure[]
{\includegraphics[width=2.5in]{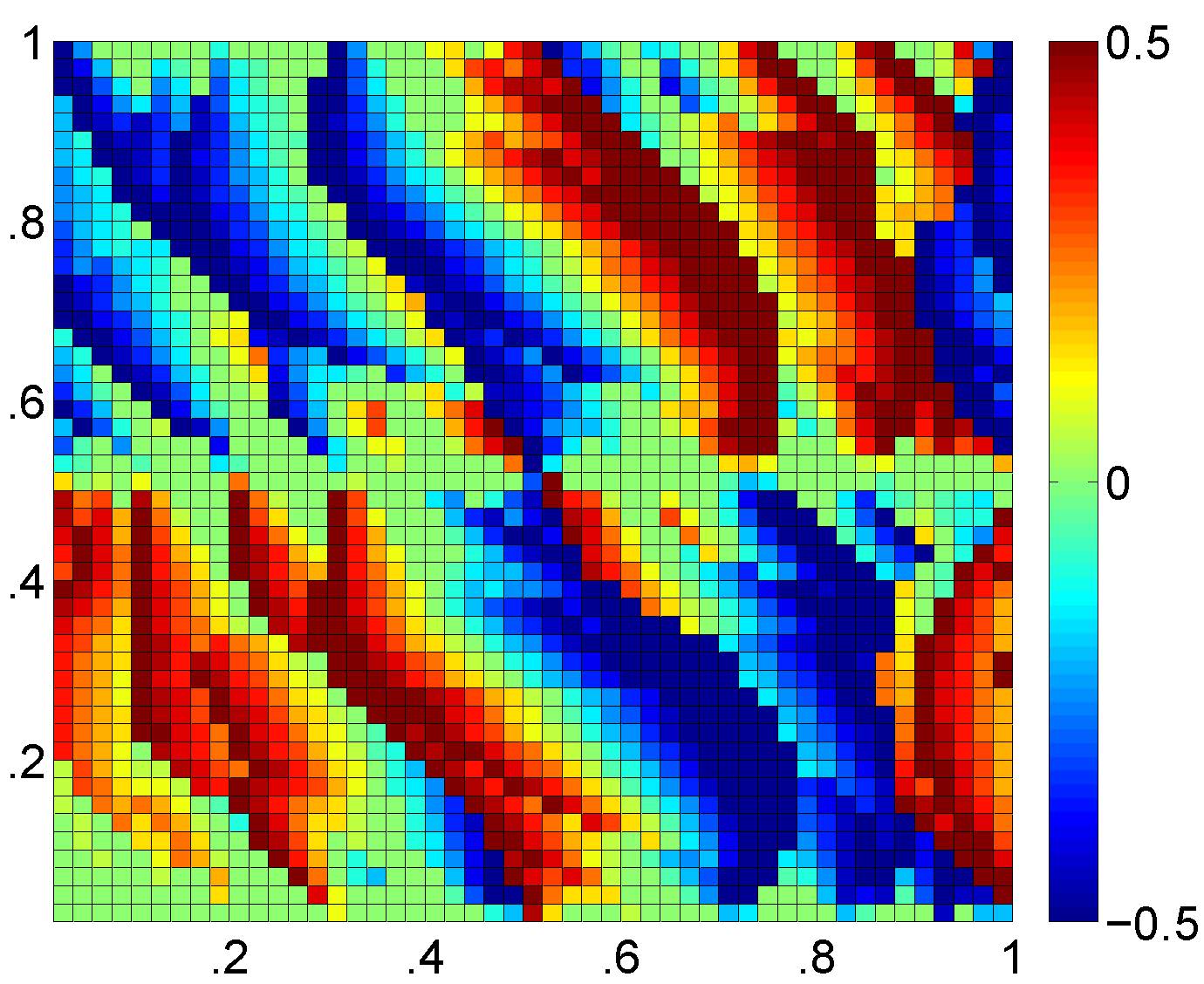}}
 \caption{a) Optimal cost function;
b) Optimal control input.} \label{stp2}
\end{figure}

\section{Conclusions}\label{conclusion}
\input{conc_ocp}
\section{Acknowledgments}
The authors would like to acknowledge Amit Diwadkar from Iowa State University for help with the simulation. We thank an anonymous referee of the previous version of the paper for improving the quality of the paper. This research work was supported by NSF Grant \# CMMI-0807666 and ECCS-1002053.

\bibliographystyle{IEEEtran}
\bibliography{ref}

%\end{IEEEbiography}
%% insert where needed to balance the two columns on the last page with
%%\newpage


%\vfill



\end{document}

%% file: abstract_ocp.tex
Numerical solutions for the optimal feedback stabilization of
discrete time dynamical systems is the focus of this paper.
Set-theoretic notion of almost everywhere stability introduced
by the Lyapunov measure, weaker than conventional
Lyapunov function-based stabilization methods, is used for
optimal stabilization. The linear Perron-Frobenius transfer operator
is used to pose the optimal stabilization problem as an infinite
dimensional linear program. Set-oriented numerical methods are
used to obtain the finite dimensional approximation of the linear
program. We provide conditions for the existence of stabilizing feedback
controls and show  the optimal stabilizing feedback control
can be obtained as a solution of a finite dimensional linear program.
The approach
is demonstrated on stabilization of period two orbit in a controlled standard map.

%% file: conc_ocp.tex
Lyapunov measure is used for the optimal stabilization of an
attractor set for a discrete time dynamical system. The optimal
stabilization problem using a Lyapunov measure is posed as an
infinite dimensional linear program. A computational framework based
on set oriented numerical methods is proposed for the finite
dimensional approximation of the linear program.
%An important feature
%of the proposed solution for the finite dimensional linear program
%is that deterministic feedback control is obtained for the optimal
%stabilization problem.

The set-theoretic notion of a.e. stability introduced by
the Lyapunov measure offers several advantages for the problem of
stabilization. First, the controller designed using the Lyapunov
measure exploits the natural dynamics of the system by allowing
the existence of unstable dynamics in the complement of the
stabilized set.
% The coarse stability introduced as a consequence
%of the finite dimensional approximation of the Lyapunov measure
%takes the notion of a.e stability one step further by allowing the
%possibility for the existence of stable dynamics with small (size
%of the partition) domain of attraction.
Second, the Lyapunov
measure provides a systematic framework for the problem of
stabilization and control of a system with complex non-equilibrium
behavior.

%
%
%The problem of optimal stabilization for discrete time nonlinear
%system is solved using linear transfer operator and Lyapunov
%measure based framework. Duality between Perron-Frobenius and
%Koopman operators is used to pose the primal and dual optimal
%stabilization problem as a infinite dimensional linear program.
%Computational framework based on set oriented numerical methods is
%used for the finite dimensional approximation of the optimal
%stabilization problem. This finite dimensional approximation of
%the optimal stabilization problem lead to solving finite number of
%linear inequalities. The highlight of the solution approach for
%the finite dimensional linear program is that the controller
%obtained is deterministic although the approximation of the linear
%transfer operators is stochastic. Simulation results for optimal
%stabilization of one dimensional maps are presented. Optimal
%stabilization using Lyapunov measure shows interesting
%characteristic of the controller for exploiting the natural
%dynamics of the uncontrolled system.

%% file: ocp_arxiv.bbl
\begin{thebibliography}{10}
\providecommand{\url}[1]{#1}
\csname url@samestyle\endcsname
\providecommand{\newblock}{\relax}
\providecommand{\bibinfo}[2]{#2}
\providecommand{\BIBentrySTDinterwordspacing}{\spaceskip=0pt\relax}
\providecommand{\BIBentryALTinterwordstretchfactor}{4}
\providecommand{\BIBentryALTinterwordspacing}{\spaceskip=\fontdimen2\font plus
\BIBentryALTinterwordstretchfactor\fontdimen3\font minus
  \fontdimen4\font\relax}
\providecommand{\BIBforeignlanguage}[2]{{%
\expandafter\ifx\csname l@#1\endcsname\relax
\typeout{** WARNING: IEEEtran.bst: No hyphenation pattern has been}%
\typeout{** loaded for the language `#1'. Using the pattern for}%
\typeout{** the default language instead.}%
\else
\language=\csname l@#1\endcsname
\fi
#2}}
\providecommand{\BIBdecl}{\relax}
\BIBdecl

\bibitem{vinterpaper}
R.~B. Vinter, ``{Convex Duality and Nonlinear Optimal Control},'' \emph{{SIAM J
  Control and Optimization}}, vol.~31, pp. 518--538, 1993.

\bibitem{hedlundrantzer}
S.~Hedlund and A.~Rantzer, ``{Convex Dynamic Programming for Hybrid Systems},''
  \emph{{IEEE Transactions on Automatic Control}}, vol.~47, no.~9, pp.
  1536--1540, 2002.

\bibitem{prajnarantzer}
S.~Prajna and A.~Rantzer, ``{Convex Programs for Temporal Verification of
  Nonlinear Dynamical Systems},'' \emph{{SIAM Journal on Control and
  Optimization}}, vol.~46, no.~3, pp. 999--1021, 2007.

\bibitem{vanhandel}
R.~Van~Handel, ``{Almost global stochastic stability},'' \emph{{SIAM Journal on
  Control and Optimization}}, vol.~45, pp. 1297--1313, 2006.

\bibitem{LassHernBook}
O.~Hern\'{a}ndez-Lerma and J.~B. Lasserre, \emph{Discrete-time Markov Control
  Processes: Basic Optimality Criteria}.\hskip 1em plus 0.5em minus 0.4em\relax
  Springer-Verlag, New York, 1996.

\bibitem{approxInfLP}
------, ``Approximation schemes for infinite linear programs,'' \emph{SIAM J.
  Optimization}, vol.~8, no.~4, pp. 973--988, 1998.

\bibitem{Grune_04}
L.~Gr{\"u}ne, ``Error estimation and adaptive discretization for the discrete
  stochastic {Hamilton-Jacobi-Bellman} equation,'' \emph{Numerische
  Mathematik}, vol.~99, pp. 85--112, 2004.

\bibitem{cell-cell1}
L.~G. Crespo and J.~Q. Sun, ``Solution of fixed final state optimal control
  problem via simple cell mapping,'' \emph{Nonlinear dynamics}, vol.~23, pp.
  391--403, 2000.

\bibitem{Junge_Osinga}
O.~Junge and H.~Osinga, ``A set-oriented approach to global optimal control,''
  \emph{ESAIM: Control, Optimisation and Calculus of Variations}, vol.~10,
  no.~2, pp. 259--270, 2004.

\bibitem{Junge_scl_05}
L.~Gr{\"u}ne and O.~Junge, ``A set-oriented approach to optimal feedback
  stabilization,'' \emph{Systems Control Lett.}, vol.~54, no.~2, pp. 169–--180,
  2005.

\bibitem{Hernandez_ocp}
D.~Hernandez-Hernandez, O.~Hernandez-Lerma, and M.~Taksar, ``A linear
  programming approach to deterministic optimal control problems,''
  \emph{Applicationes Mathematicae}, vol.~24, no.~1, pp. 17--33, 1996.

\bibitem{Gait_ocp}
V.~Gaitsgory and S.~Rossomakhine, ``Linear programming approach to
  deterministic long run average optimal control problems,'' \emph{SIAM J.
  Control ad Optimization}, vol.~44, no.~6, pp. 2006--2037, 2006.

\bibitem{Lasserre_ocp}
J.~Lasserre, C.~Prieur, and D.~Henrion, ``Nonlinear optimal control: Numerical
  approximation via moment and {LMI}-relaxations,'' in \emph{Proceeding of IEEE
  {C}onference on {D}ecision and {C}ontrol}, Seville, Spain, 2005.

\bibitem{Meyn_sadhana}
S.~Meyn, ``{Algorithm for optimization and stabilization of controlled {M}arkov
  chains},'' \emph{{Sadhana}}, vol.~24, pp. 339--367, 1999.

\bibitem{viscosity_solnHJB_book}
M.~Bardi and I.~Capuzzo-Dolcetta, \emph{{Optimal control and viscosity
  solutions of Hamilton-Jacobi-Bellman equations}}.\hskip 1em plus 0.5em minus
  0.4em\relax Boston: Birkhauser, 1997.

\bibitem{Rantzer01}
A.~Rantzer, ``A dual to {L}yapunov's stability theorem,'' \emph{Systems \&
  Control Letters}, vol.~42, pp. 161--168, 2001.

\bibitem{Praly_conconvex}
C.~Prieur and L.~Praly, ``Uniting local and global controller,'' in
  \emph{Proceedings of IEEE Conference on Decision and Control}, AZ, 1999, pp.
  1214--1219.

\bibitem{VaidyaMehtaTAC}
U.~Vaidya and P.~G. Mehta, ``Lyapunov measure for almost everywhere
  stability,'' \emph{IEEE Transactions on Automatic Control}, vol.~53, pp.
  307--323, 2008.

\bibitem{Rajeev_continuous_time_journal}
R.~Rajaram, U.~Vaidya, M.~Fardad, and B.~Ganapathysubramanian, ``Almost
  everywhere stability: Linear transfer operator approach,'' \emph{Journal of
  Mathematical analysis and applications}, vol. 368, pp. 144--156, 2010.

\bibitem{Vaidya_CLM}
U.~Vaidya, P.~Mehta, and U.~Shanbhag, ``Nonlinear stabilization via control
  {L}yapunov measure,'' \emph{IEEE Transactions on Automatic Control}, vol.~55,
  pp. 1314--1328, 2010.

\bibitem{arvind_ocp_online}
A.~Raghunathan and U.~Vaidya, ``Optimal stabilization using {L}yapunov
  measures,'' http://www.ece.iastate.edu/{$\sim$}ugvaidya/publications.html,
  2012.

\bibitem{arvind_ocp}
------, ``Optimal stabilization using {L}yapunov measure,'' in
  \emph{Proceedings of American Control Conference}, Seattle, WA, 2008, pp.
  1746--1751.

\bibitem{Lasota}
A.~Lasota and M.~C. Mackey, \emph{Chaos, Fractals, and Noise: Stochastic
  Aspects of Dynamics}.\hskip 1em plus 0.5em minus 0.4em\relax New York:
  Springer-Verlag, 1994.

\bibitem{Vaidya_converse}
U.~Vaidya, ``Converse theorem for almost everywhere stability using {L}yapunov
  measure,'' in \emph{Proceedings of American Control Conference}, New York,
  NY, 2007.

\bibitem{disintegration}
H.~Furstenberg, \emph{Recurrence in Ergodic theory and Combinatorial Number
  Theory}.\hskip 1em plus 0.5em minus 0.4em\relax Princeston, New Jersey:
  Princeston University Press, 1981.

\bibitem{AndNash}
E.~Anderson and P.~Nash, \emph{Linear Programming in Infinite-Dimensional
  Spaces - Theory and Applications}.\hskip 1em plus 0.5em minus 0.4em\relax
  John Wiley \& Sons, Chichester, U.K., 1987.

\bibitem{MR1297120}
O.~L. Mangasarian, \emph{Nonlinear programming}, ser. Classics in Applied
  Mathematics.\hskip 1em plus 0.5em minus 0.4em\relax Philadelphia, PA: Society
  for Industrial and Applied Mathematics (SIAM), 1994, vol.~10, corrected
  reprint of the 1969 original.

\bibitem{PDIPM}
S.~J. Wright, \emph{Primal-Dual Interior-Point Methods}.\hskip 1em plus 0.5em
  minus 0.4em\relax Philadelphia, Pa: Society for Industrial and Applied
  Mathematics, 1997.

\bibitem{ipopt}
A.~W\"{a}chter and L.~T. Biegler, ``On the implementaion of a primal-dual
  interior point filter line search algorithm for large-scale nonlinear
  programming,'' \emph{Mathematical Programming}, vol. 106, no.~1, pp. 25--57,
  2006.

\bibitem{coinorrep}
\BIBentryALTinterwordspacing
{COIN-OR Repository}. [Online]. Available: \url{http://www.coin-or.org/}
\BIBentrySTDinterwordspacing

\bibitem{NocedalWright}
\BIBentryALTinterwordspacing
J.~Nocedal and S.~Wright, \emph{Numerical optimization}, ser. Springer series
  in operations research.\hskip 1em plus 0.5em minus 0.4em\relax Springer,
  2006. [Online]. Available:
  \url{http://books.google.com/books?id=eNlPAAAAMAAJ}
\BIBentrySTDinterwordspacing

\bibitem{Vaidya_control_twist}
U.~Vaidya and I.~Mezi\'{c}, ``Controllability for a class of area preserving
  twist maps,'' \emph{Physica D}, vol. 189, pp. 234--246, 2004.

\end{thebibliography}
